\documentclass[12pt]{amsart}
\usepackage{amssymb,amsthm,mathrsfs,amsmath,bm,cite}

\newtheorem{thm}{Theorem}[section]
\newtheorem{theorem}[thm]{Theorem}

\newtheorem{lemma}[thm]{Lemma}
\newtheorem{proposition}[thm]{Proposition}

\theoremstyle{definition}

\theoremstyle{remark}
\newtheorem{remark}[thm]{Remark}

\newenvironment{theorem*}[1]{\smallskip\noindent{\bf #1.}\it}{\medskip}

\newcommand\dom{\operatorname{dom}}

\newcommand\bC{{\mathbb C}}

\newcommand\bR{{\mathbb R}}
\newcommand\bZ{{\mathbb Z}}

\newcommand\al{\alpha}
\newcommand\la{\lambda}

\newcommand\bu{\mathbf{u}}
\newcommand\bw{\mathbf{w}}
\newcommand\bv{\mathbf{v}}

\numberwithin{equation}{section} 
\begin{document}
\title{Asymptotics of eigenvalues and eigenfunctions of energy-dependent Sturm--Liouville equations}

\author[N.~Pronska]{Nataliya Pronska}%

\address{Institute for Applied Problems of Mechanics and Mathematics,
3b~Naukova st., 79601 Lviv, Ukraine}
\email{nataliya.pronska@gmail.com}



\date{\today}

\begin{abstract}
We study asymptotics of eigenvalues, eigenfunctions and norming
constants of singular energy-dependent Sturm--Liouville equations
with complex-valued potentials. The analysis essentially exploits
the integral representation of solutions, which we derive using
the connection of the problem under study and a Dirac system of a
special form.
\end{abstract}

\maketitle

\section{Introduction}

The main objective of the present paper is to study asymptotics of
eigenvalues and eigenfunctions of Sturm--Liouville equations
on~(0,1) with energy-dependent potentials, viz.
\begin{equation}\label{eq:intr.spr}
    -y''+qy+2\lambda p y=\lambda^2y.
\end{equation}
Here~$\la\in\bC$ is the spectral parameter,~$p$ is a
complex-valued function in~$L_2(0,1)$ and~$q$ is a complex-valued
distribution in the Sobolev space~$W_2^{-1}(0,1)$, i.e.~$q =r'$
with a complex-valued~$r\in L_{2}(0,1)$. We consider
equation~\eqref{eq:intr.spr} mostly under the Dirichlet boundary
conditions
\begin{equation}\label{eq:intr.b.c.}
y(0)=y(1)=0.
\end{equation}
We restrict our attention to such boundary conditions just to
concentrate on the ideas and to avoid unnecessary technicalities.
Other separated boundary conditions can be treated analogously; in
particular, in the last section we formulate some results for the
case of the mixed boundary conditions.

Energy-dependent Sturm--Liouville equations are of importance in
classical and quantum mechanics. For instance, they are used for
modelling mechanical systems vibrations in viscous media. The
Klein--Gordon equations, which describe the motion of massless
particles such as photons, can also be reduced to the
form~\eqref{eq:intr.spr}. The corresponding evolution equations
are used to model the interactions between colliding relativistic
spinless particles. In such mechanical models the spectral
parameter~$\la$ is related to the energy of the system, which
explains the terminology ``energy-dependent'' used for the
spectral equation~\eqref{eq:intr.spr}.

Asymptotic behaviour of eigenvalues, eigenfunctions, and other
spectral characteristics for usual Sturm--Liouville operators are
studied sufficiently well (see, e.g.\
\cite{LevSar:1991,KosSar:79,PosTru:1987}). The spectral problem
for energy-dependent Sturm--Liouville equation~\eqref{eq:intr.spr}
with~$p\in W_2^1[0,\pi]$ and~$q\in L_2[0,\pi]$ was considered by
M.~Gasymov and G.~Guseinov in their short paper~\cite{GasGus81} of
1981. Some results on asymptotics are formulated there without
proofs. Analogous problems under more general boundary conditions
were also considered by I.~Nabiev in \cite{Nab041}. The
asymptotics of eigenvalues and eigenfunctions for Sturm--Liouville
equations with singular potentials~$q\in W_2^{-1}(0,1)$ were
studied by A.~Savchuk and A.~Shkalikov in~\cite{SavShk:1999}.

Our aim in this paper is to investigate the corresponding
asymptotics for energy-dependent Sturm--Liouville
equations~\eqref{eq:intr.spr} under minimal smoothness assumptions
on the potentials~$p$ and~$q$, including, e.g.\ the case when~$q$
contains Dirac delta-functions and/or Coulumb-like singularities.
Our approach consists in establishing a strong connection between
the spectral problem~\eqref{eq:intr.spr} and the spectral problem
for a Dirac system of a special form. We then study the latter
and, in particular, derive the integral representation for the
solution of~\eqref{eq:intr.spr} which, in turn, is a basis for the
subsequent asymptotic analysis

The paper is organised as follows. In the next section, we
rigorously set the spectral problem under study and give main
definitions. Connection between the spectral
problem \eqref{eq:intr.spr} and that for a special Dirac system is
discussed in Section~\ref{sec:red}. In Section~\ref{sec:tr.op}, we
construct the transformation operator relating the solution of the
obtained system with that of the Dirac system with zero potential.
Next in Section~\ref{sec:asy}, we derive the asymptotics of
eigenvalues, eigenfunctions and the corresponding norming
constants for the
problem~\eqref{eq:intr.spr},~\eqref{eq:intr.b.c.} and justify the
factorization formula for its characteristic function.  In the
last section we formulate analogous results for the spectral
problem~\eqref{eq:intr.spr} under more general boundary
conditions. Appendix~\ref{sec:A} contains main definitions from
the spectral theory for operator pencils. We also prove there that
the algebraic multiplicity of~$\la$ as an eigenvalue of the
problem~\eqref{eq:intr.spr},~\eqref{eq:intr.b.c.} coincides with
that of~$\la$ as an eigenvalue of the corresponding operator
pencil.

\emph{Notations.} Throughout the paper, we denote by
$\mathcal{M}_2=\mathcal{M}_2(\bC)$ the linear space of $2\times 2$
matrices with complex entries endowed with the Euclidean operator
norm. Next,~$p_0$ will stand for~$\int_0^1p(s)ds$. The
superscript~$\mathrm{t}$ will signify the transposition of vectors
and matrices, e.g.\ $(c_1,c_2)^{\mathrm{t}}$ is the column
vector~$\binom{c_1}{c_2}$.

\section{Preliminaries}\label{sec:pre}

In this section, we recall main definitions and formulate the
spectral problem under study more rigorously. To start with,
introduce the differential expression
\[
    \ell(y):=-y''+qy.
\]
Since the potential~$q$ is a complex-valued distribution, we need
to define the action of~$\ell$ in detail. To do this, we use the
regularization by quasi-derivative method due to Savchuk and
Shkalikov~\cite{SavShk:1999,SavShk:2003}. Take a function~$r$
from~$L_2(0,1)$ such that~$q=r'$ and for every absolutely
continuous~$y$ introduce its quasi-derivative~$y^{[1]}:=y'-ry$.
Then define~$\ell(y)$ as
\[
   \ell (y) = -\bigl(y^{[1]}\bigr)' - r
y^{[1]} - r^2 y
\]
on the domain
\[
    \dom \ell = \{y \in AC [0,1] \mid y^{[1]} \in AC[0,1], \ \ell(y) \in L_2(0,1)\}.
\]
A straightforward verification shows that so defined~$\ell(y)$
coincides with~$-y''+qy$ in the distributional sense. Therefore we
can recast the equation~\eqref{eq:intr.spr} as
\begin{equation}\label{eq:pre.spr}
    \ell(y)+2\la p y=\la^2y.
\end{equation}

A number $\lambda\in\bC$ is called an \emph{eigenvalue} of the
problem~\eqref{eq:intr.spr},~\eqref{eq:intr.b.c.} if
equation~\eqref{eq:pre.spr} possesses a nontrivial solution
satisfying the boundary conditions~\eqref{eq:intr.b.c.}. This
solution is then called an \emph{eigenfunction} of the
problem~\eqref{eq:intr.spr},~\eqref{eq:intr.b.c.} corresponding
to~$\la$.

Let~$y(x,z)$ be the solution of~\eqref{eq:pre.spr} with~$z$
instead of~$\la$ subject to the initial
conditions~$y(0)=0$,~$y^{[1]}(0)=1$. This solution exists and is
unique~\cite{SavShk:1999}, so that~$\la$ is an eigenvalue of the
problem~\eqref{eq:intr.spr},~\eqref{eq:intr.b.c.} if and only if
it is a zero of the \emph{characteristic
function}~$\varphi(z):=y(1,z)$. The corresponding eigenfunction
then coincides up to a constant factor with~$y(\cdot,\la)$. The
multiplicity of~$\la$ as a zero of~$\varphi(z)$ is called an
\emph{algebraic multiplicity} of the eigenvalue~$\la$
of~\eqref{eq:intr.spr},~\eqref{eq:intr.b.c.}.

As we shall see further~$\varphi$ is an analytic nonconstant
function, and so the set of its zeros is a discrete subset
of~$\bC$. This shows that the set of eigenvalues
of \eqref{eq:intr.spr},\eqref{eq:intr.b.c.} is discrete. Without
loss of generality, we shall make the following standing
assumption:
\begin{itemize}
\item[(A)] $0$ is not a zero of the characteristic function~$\varphi(z)$, i.e.\ it is not an eigenvalue
of the problem~\eqref{eq:intr.spr},~\eqref{eq:intr.b.c.}.
\end{itemize}
In fact, (A) is achieved  by shifting the spectral parameter~$\la$
if necessary; then for~$\la_0$ such that~$\varphi(\la_0)\ne 0$ the
problem~\eqref{eq:intr.spr},~\eqref{eq:intr.b.c.} with new
spectral parameter~$\mu:=\la-\la_0$ and with~$p$ and~$q$ replaced
by~$p+\la_0$ and~$q-2\la_0p-\la_0^2$ respectively the
assumption~(A) holds.

 The
spectral problem~\eqref{eq:intr.spr},~\eqref{eq:intr.b.c.} can be
regarded as the spectral problem for the quadratic operator
pencil~$T$ (see~\cite{Mar:88,Pro:2012pro}) defined by
\begin{equation}\label{eq:intr.T}
    T(\la)y:= \la^2y-2\la py +y''-qy
\end{equation}
on the~$\la$-independent domain
\[
    \dom T:=\{y \in \dom \ell \mid y(0)=y(1)=0\}.
\]
For the pencil~$T$ one can introduce the notions of the spectrum,
the eigenvalues and corresponding eigenvectors, their geometric
and algebraic multiplicities (see, e.g. \cite{Mar:88} and
Appendix~\ref{sec:A}). The spectral properties of~$T$ were
discussed in~\cite{Pro:2012pro}. In particular, it was proved
therein that the spectrum of~$T$ consists only of eigenvalues,
which can easily be shown to coincide with the eigenvalues of the
problem~\eqref{eq:intr.spr},~\eqref{eq:intr.b.c.} defined above.
In Appendix~\ref{sec:A}, we show that the algebraic multiplicity
of~$\la$ as an eigenvalue
of~\eqref{eq:intr.spr},~\eqref{eq:intr.b.c.} coincides with that
of~$\la$ as an eigenvalue of~$T$.

\section{Reduction to the Dirac system}\label{sec:red}

In this section we reduce equation~\eqref{eq:intr.spr} to
a~$\la$-linear Dirac-type system of the first order. We shall
further use the connection between~\eqref{eq:intr.spr} and this
system to derive the asymptotics of interest.

The following observation plays an important role in the reduction
procedure.

\begin{lemma}\label{lem:pos_sol}
The equation~$\ell(y)=0$ possesses a complex-valued solution which
does not vanish on~$[0,1]$.
\end{lemma}
\begin{proof}
Note firstly that for every complex~$a,b$ and every~$x_0$
from~$[0,1]$ the equation~$\ell(y)=0$ possesses a unique solution
satisfying the conditions~$y(x_0)=a$ and~$y^{[1]}(x_0)=b$ (see,
e.g.\ \cite{SavShk:1999}).

Assume that~$y$ is a solution of~$\ell(y)=0$. We introduce the
polar coordinates~$\rho$ and~$\theta$ via~$y(x)=\rho(x)\sin
\theta(x)$ and~$y^{[1]}(x)=\rho(x)\cos\theta(x)$. Clearly, the
solution~$y$ vanishes at some point~$x_0$ from~$[0,1]$ if and only
if~$\theta(x_0)=\pi k$ for some~$k\in\bZ$.  The function~$\theta$
is called the Pr\"{u}fer angle (see~\cite{Har:1964,SavShk:2003})
and can be defined to be continuous; it then satisfies the
equation
\begin{equation}\label{eq:pre.eq_theta}
    \theta'(x)=(\cos \theta(x)+r(x)\sin\theta(x))^2.
\end{equation}

Equation~\eqref{eq:pre.eq_theta} has the
form~$\theta'=f(x,\theta)$ with the right-hand side~$f$ that is
not continuous in~$x$. However,~\eqref{eq:pre.eq_theta} is a
so-called Caratheodory equation and under the initial
condition~$\theta(\xi)=0$ with~$\xi\in[0,1]$ it possesses a unique
solution~$\theta(x,\xi)$ (see~e.g.~\cite[Theorem 1.1.2]{Fil:88});
this solution depends continuously on~$\xi$ (see~\cite[Theorem
2.8.2]{Fil:88}). Therefore the mapping~$\xi\mapsto \theta(0,\xi)$
is continuous and its image~$I_0$ is a compact subset of~$\bC$
containing~$0$ as a continuious image of a compactum~$[0,1]$. Note
further that for any~$k\in\bZ$ the solutions~$\theta_k(x,\xi)$ of
the problem~\eqref{eq:pre.eq_theta} with~$\theta(\xi)=\pi k$ are
equal to~$\theta_k(x,\xi)=\theta(x,\xi)+\pi k$. The images~$I_k$
of the mappings~$\xi\mapsto\theta_k(x,\xi)$ are compact and are
the shifts of~$I_0$ by~$\pi k$ along the real axis.

Let us now take any complex number, say~$\theta_0$, outside the
union of all the compacta~$I_k$,~$k\in\bZ$, and consider the
problem~\eqref{eq:pre.eq_theta} with~$\theta(0)=\theta_0$. In view
of the above arguments and uniqueness of the solution of the
corresponding initial-value problem, the solution
of~\eqref{eq:pre.eq_theta} with~$\theta(0)=\theta_0$ can equal
to~$\pi k$, $k\in\bZ$, at no point of the interval~$[0,1]$.
Consider the solution~$y_0$ of~$\ell(y)=0$ subject to the initial
conditions~$y(0)=1$, $y^{[1]}(0)=\cot \theta_0$. Then~$y_0$ does
not vanish on~$[0,1]$, which is the assertion of the lemma.
\end{proof}

Denote by~$y_0$ any solution of~$\ell(y)=0$ not vanishing
on~$[0,1]$ and set~$v=y_0'/y_0$. Observe that~$v\in L_2(0,1)$
and~$q=v'+v^2$,~i.e.~$q$ is a Miura potential
(see~\cite{KapPerShuTop:2005}). Then the differential
expression~$-y''+qy$ can be written in the factorized form, viz.
\begin{equation}\label{eq:Dir.Afact}
   -y''+qy  = -\Bigl(\frac{d}{dx}+v\Bigr)\Bigl(\frac{d}{dx}-v\Bigr)y.
\end{equation}

\begin{remark}\label{rem:v-r_cont}
Observe that the function~$v$ satisfies the
equality~$v-r=y_0^{[1]}/y_0$, so that~$v-r$ is a continuous
function on~$[0,1]$ and~$(v-r)(x)=\cot \theta(x)$ for
every~$x\in[0,1]$, where~$\theta$ is the Pr\"{u}fer angle
corresponding to~$y_0$.
\end{remark}

For~$\la\ne 0$ consider the functions~$u_2:=y$
and~$u_1:=(y'-vy)/\la$ and recast the spectral
equation~\eqref{eq:intr.spr} as the following first order system
for~$u_1$ and~$u_2$:
\begin{align}\label{eq:Dir.system1}
    u_2' - v u_2 &= \la u_1,\\
    -u_1'-v u_1 + 2 p u_2 &= \la u_2. \label{eq:Dir.system2}
\end{align}
Setting
\begin{equation}\label{eq:Dir.P}
    J:=\left(
           \begin{array}{cc}
             0 & 1 \\
             -1 & 0 \\
           \end{array}
         \right),
\qquad
    P:=\left(
                \begin{array}{cc}
                  0 & -v \\
                  -v & 2p \\
                \end{array}
              \right),
\qquad
    \bu(x)=\binom{u_1}{u_2},
\end{equation}
we see that the above system is the spectral problem
$\ell(P)\bu=\la\bu$ for a Dirac differential expression $\ell(P)$
acting in $L_2(0,1)\times L_2(0,1)$ via
\begin{equation}
\label{eq:dif.exp.Dirac}
    \ell(P)\bu: =J\frac{d\bu}{dx}+P\bu
\end{equation}
on the domain
\[
    \dom \ell(P)=\{\bu=(u_1,u_2)^\mathrm{t}\mid \bu\in W_2^1(0,1)\times
    W_2^1(0,1)\}.
\]

It was shown in~\cite{Pro:2011c} that the spectral
problem~\eqref{eq:intr.spr},~\eqref{eq:intr.b.c.} is closely
related to the spectral problem for the Dirac
operator~$\mathcal{D}(P)$ defined by the differential
expression~$\ell(P)$ on the domain
\[
\dom \mathcal{D}(P)=\{\bu=(u_1,u_2)^\mathrm{t}\mid
\bu\in\dom\ell(P),\; u_2(0)=u_2(1)=0\}.
\]
In particular, the nonzero spectra for both problems coincide
counting with multiplicity. Moreover,
~$\bu=(u_1,u_2)^{\mathrm{t}}$ is an eigenfunction of the
operator~$\mathcal{D}(P)$ corresponding to the eigenvalue~$\la\ne
0$ if and only if $y=u_2$ is an eigenfunction of
\eqref{eq:intr.spr},~\eqref{eq:intr.b.c.} corresponding to~$\la$
and~$u_1=(u_2'-vu_2)/\la$.

By assumption~(A),~$\la=0$ is not an eigenvalue
of~\eqref{eq:intr.spr},~\eqref{eq:intr.b.c.}; however, it is in
the spectrum of~$\mathcal{D}(P)$:

\begin{lemma}\label{lem:red.mult_0}
Under assumption~(A),~$\la=0$ is an eigenvalue of~$\mathcal{D}(P)$
of algebraic multiplicity one.
\end{lemma}
\begin{proof} A straightforward verification shows that~$\la=0$ is an eigenvalue
of~$\mathcal{D}(P)$, and that every corresponding eigenfunction
must be collinear to~$\bu_0=(u_1,u_2)^{\mathrm t}$, where
$u_1=\exp\{-\int v\}$,~$u_2\equiv 0$. Therefore the
eigenvalue~$\la=0$ is geometrically simple.

Suppose that the algebraic multiplicity of~$\la=0$ is greater than
one. Then there exists a vector~$\bw=(w_1,w_2)^{\mathrm t}$ from
the domain of~$\mathcal{D}(P)$ associated with~$\bu_0$,~i.e.
satisfying the equality~$\mathcal{D}(P)\bw=\bu_0$.
Then~$\left(\frac{d}{dx}-v\right)w_2= u_1$ and thus~$\ell
w_2=-\left(\frac{d}{dx}+v\right)\left(\frac{d}{dx}-v\right)w_2=-\left(\frac{d}{dx}+v\right)u_1=0$.
Moreover,~$\bw\in\dom \mathcal D(P)$ requires
that~$w_2(0)=w_2(1)=0$ and thus either~$w_2$ is an eigenfunction
of~\eqref{eq:intr.spr},~\eqref{eq:intr.b.c.} for the
eigenvalue~$\la=0$ or~$\bw\equiv 0$. The first possibility is
ruled out by assumption~(A), while the second one is impossible in
view of the relation~$w_2'-vw_2=u_1$. The contradiction derived
shows that no such~$\bw$ exists and finishes the proof.
\end{proof}


\section{Transformation operator}\label{sec:tr.op}


In this section we construct the so-called transformation operator
relating the solution of the system~$\ell(P)\bu=\la\bu$ and that
of~$\ell(P_0)\bu=\la\bu$ with zero matrix potential~$P_0$ (i.e.
with matrix potential having all components zero).

Denote  by~$U(x,\la)$ a $2\times 2$ matrix-valued function
satisfying the equation
\begin{equation}\label{eq:Dirac_eq}
J\frac{dU}{dx}+PU=\la U
\end{equation}
and the initial condition~$U(0)=I$.

\begin{theorem}\label{thm:tr.op}
Let~$P$ in~\eqref{eq:Dirac_eq} be of the form~\eqref{eq:Dir.P}
with~$p$ and~$v$ from~$L_2(0,1)$. Then
\begin{equation}\label{eq:U}
U(x,\la)=e^{a(x)J}+\int_0^xe^{-\la(x-2s)J}K(x,s)ds,
\end{equation}
where~$a(x)=a(x,\la)=\int_0^xp(s)ds-\la x$ and~$K$ is a
matrix-valued function such that for every~$x\in[0,1]$ the
function~$K(x,\cdot)$ is from~$L_2((0,1),\mathcal{M}_2)$.
Moreover, the mapping
\begin{equation}\label{eq:tr_op.map}
    x\mapsto K(x,\cdot)\in L_2((0,1),\mathcal{M}_2)
\end{equation}
is continuous on~$[0,1]$.
\end{theorem}
This theorem is very similar to the corresponding theorem
of~\cite{AlbHryMk:2005:RJMP} and its proof requires only minor
modifications.
\begin{proof}
Observe that the system~\eqref{eq:Dirac_eq} can be rewritten as
\[
J\frac{dU}{dx}+QU=(\la-p) U
\]
with
\[Q=\left(%
\begin{array}{cc}
  -p & -v \\
  -v & p \\
\end{array}%
\right)=pJ_1-v J_2, \quad J_1=\left(%
\begin{array}{cc}
  -1 & 0 \\
  0 & 1 \\
\end{array}%
\right),\quad J_2=\left(%
\begin{array}{cc}
  0 & 1 \\
  1 & 0 \\
\end{array}%
\right).
\]
The variation of constants method shows that~$U$ satisfies the
following integral equation:
\[
    U(x)=e^{a(x)J}+\int\limits_0^xe^{(a(x)-a(s))J}JQ(s)U(s)ds.
\]
This equation can be solved by the method of successive
approximation. Setting
\begin{equation}\label{eq:rec}
    U_0(x)=e^{a(x)J},\quad
U_n(x)=\int\limits_0^xe^{(a(x)-a(s))J}JQ(s)U_{n-1}(s)ds,
\end{equation}
we see that the solution of the above equation can formally be
given by the sum~$\sum_{n=0}^{\infty}U_n$. We shall prove below
that
\begin{equation}\label{eq:conv}
    \sum_{n=0}^{\infty}\|U_n\|_\infty< \infty
\end{equation}
(here~$\|U_n\|_\infty:=\sup_{x\in[0,1]}|U_n(x)|$, and~$|U_n(x)|$
is the Euclidean norm of the matrix~$U_n(x)$) whence the
series~$\sum_{n=0}^{\infty}U_n$ converges in the
space~$L_{\infty}([0,1],\mathcal{M}_2)$ to the solution of the
equation~\eqref{eq:Dirac_eq} which satisfies the initial
condition~$U(0)=I$.

Let us now prove~\eqref{eq:conv}.
Set~$\tilde{Q}(y):=\exp\{-2\int_0^tpJ\}JQ(t)$
and
\[
\mathcal{Q}_n(t_1,\dots,t_n):=\tilde{Q}(t_n)\tilde{Q}(t_{n-1})\dots
\tilde{Q}(t_1).
\]
Observe firstly that the matrix~$Q$
anti-commutes with~$J$ and therefore
\[
    e^{-tJ}\tilde{Q}(s)=\tilde{Q}(s)e^{tJ}.
\]
Using this in the relations~\eqref{eq:rec}, we obtain by induction
that
\[
    U_n(x)=e^{\int_0^xp(s)dsJ}\int\limits_{\Pi_n(x)}e^{-\la(x-2\xi_n(\mathrm{t}))J}\mathcal{Q}_n(t_1,\dots,t_n)dt_1\dots
    dt_n,
\]
where
\begin{align*}
\Pi_n(x)&=\{\mathrm{t}:=(t_1,\dots,t_n)\in \bR^n\mid 0\le
t_1\le\dots\le t_n\le x\},\\
\xi_n(\mathrm{t})&=\sum\limits_{j=1}^n(-1)^{n-j}t_j.
\end{align*}
Setting~$s=\xi_n(\mathrm{t})$, we can rewrite the equality
for~$U_n$ as
\[
    U_n(x)=\int_{0}^xe^{-\la(x-2s)J}K_n(x,s)ds,
\]
where~$K_1(x,s)\equiv e^{\int_0^xpJ}\tilde{Q}(s)$ and for~$n\ge 2$
\[
    K_n(x,s)=e^{\int_0^xpJ}\int\limits_{\Pi_{n-1}^*(x,s)}\!\mathcal{Q}_n(t_1,\dots,t_{n-1},s+\xi_{n-1}(\mathrm{t}))dt_1\dots dt_{n-1}
\]
with~$0\le s<x\le 1$ and
\[
    \Pi_{n-1}^*(x,s)=\{\mathrm{t}:=(t_1,\dots,t_{n-1})\in \bR^{n-1}\mid 0\le
t_1\le\dots\le t_{n-1}\le \xi_{n-1}(\mathrm{t})+s\le x\}.
\]
Let us estimate the~$L_2$-norm of~$K_n(x,\cdot)$:
\begin{align*}
    \|K_n(x,\cdot)\|_2^2&=\int\limits_0^1|K_n(x,s)|^2ds\\&\le
    \frac{1}{(n-1)!}\int\limits_0^1ds\!\int\limits_{\Pi_{n-1}^*(x,s)}\!e^{2\int\limits_0^x|\mathrm{Im} p|}|\mathcal{Q}_n(t_1,\dots,t_{n-1},s+\xi_{n-1}(\mathrm{t}))|^2dt_1\dots dt_{n-1}\\
    &\le\frac{1}{(n-1)!}\int\limits_{\Pi_n(x)}e^{2\int\limits_0^x|\mathrm{Im} p|}|\mathcal{Q}_n(t_1,\dots,t_n)|^2dt_{1}\dots dt_{n}\\
    &=\frac{1}{((n-1)!)^2n}e^{2\int\limits_0^x|\mathrm{Im}
    p|}\Biggl(\int\limits_0^x|\tilde{Q}|^2\Biggr)^n\le\frac{e^{2\|p\|_1}\|\tilde{Q}\|^{2n}_2}{((n-1)!)^2}
\end{align*}
where~$\|K(\cdot)\|_2:=\left(\int_0^1|K(s)|^2ds\right)^{1/2}$
and~$|K(x)|$ is the Euclidean norm of the matrix~$K(x)$.
Put~$C:=\mathrm{max}_{x\in[-1,1]}|e^{-\la x J}|$. Then
\[
    |U_n(x)|\le C \int\limits_0^x|K_n(x,s)|ds\le C\|K_n(x,\cdot)\|_2\le C \frac{e^{\|p\|_1}\|\tilde{Q}\|^{n}_2}{(n-1)!},
\]
which yields~\eqref{eq:conv}. The estimate of the
norm~$\|K_n(x,\cdot)\|_2$ implies also the convergence of the
series~$K(x,\cdot):=\sum_{n=1}^\infty K_n(x,\cdot)$
in~$L_2((0,1),\mathcal{M}_2)$ with
\[
    \|K(x,\cdot)\|_2\le \sum\limits_{n=1}^\infty \frac{e^{\|p\|_1}\|\tilde{Q}\|^n_2}{(n-1)!}
    \le \|\tilde{Q}\|_2 \exp\{\|\tilde{Q}\|_2+\|p\|_1\}.
\]
Therefore the first statement of the theorem is proved.

To prove continuity of~\eqref{eq:tr_op.map} it is enough to verify
continuity of the mapping~$x\mapsto
\exp\Bigl\{-\int\limits_0^xpJ\Bigr\}K(x,\cdot)=:\tilde{K}(x,\cdot)$.
Take~$x_1,x_2 \in [0,1]$ such that~$x_1<x_2$.
Set~$\tilde{K}_n(x,s):=\exp\Bigl\{-\int\limits_0^xpJ\Bigr\}K_n(x,\cdot)$;
then
\[
    \tilde{K}_n(x_2,s)-\tilde{K}_n(x_1,s)=
    \!\int\limits_{\Pi_{n-1}^{**}(x_1,x_2,s)}\!\mathcal{Q}_n(t_1,\dots,t_{n-1},s+\xi_{n-1}(\mathrm{t}))dt_1\dots
    dt_{n-1},
\]
where
\begin{align*}
    \Pi_{n-1}^{**}(x_1,x_2,s):=\{\mathrm{t}:=(t_1,\dots,t_{n-1})\in \bR^{n-1}\mid & \;0\le
t_1\le\dots\le t_{n-1}\le \xi_{n-1}(\mathrm{t})+s,\\& x_1\le
\xi_{n-1}(\mathrm{t})+s\le x_2\}.
\end{align*}
Therefore,
\begin{align*}
    \int\limits_0^1|\tilde{K}_n(x_2,s)-&\tilde{K}_n(x_1,s)|^2ds\\\le&
    \frac{1}{(n-1)!}\int\limits_0^1\!\int\limits_{\Pi_{n-1}^{**}(x_1,x_2,s)}\!|\mathcal{Q}_{n}(t_1,\dots,t_{n-1},s+\xi_{n-1}(\mathrm{t})|^2dt_1\dots
    dt_{n-1}ds\\ \le&
    \frac{1}{(n-1)!}\int\limits_{x_1}^{x_2}dt_n\int\limits_{\Pi_{n-1}(x_2)}|\mathcal{Q}_n(t_1,\dots,t_n)|^2dt_1\dots
    dt_{n-1}\\ \le&
    \frac{1}{((n-1)!)^2}\|\tilde{Q}\|_2^{2(n-1)}\cdot\int\limits_{x_1}^{x_2}|\tilde{Q}(t)|^2dt.
\end{align*}
This yields the estimate of the norms
\begin{equation*}\label{eq:est}
    \|\tilde{K}_n(x_2,\cdot)-\tilde{K}_n(x_1,\cdot)\|_2\le\frac{1}{(n-1)!}\|\tilde{Q}\|_2^{n-1}\left[\int\limits_{x_1}^{x_2}|\tilde{Q}(t)|^2dt\right]^{1/2}
\end{equation*}
and so
\[
    \|\tilde{K}(x_1,\cdot)-\tilde{K}(x_2,\cdot)\|\le C\left[\int\limits_{x_1}^{x_2}|\tilde{Q}(t)|^2dt\right]^{1/2}\exp\{\|\tilde{Q}\|\}
\]
with some constant~$C$ depending on~$p$. This shows that  the
mapping~$x\mapsto \tilde{K}(x,\cdot)$ is continuous from~$[0,1]$
to~$L_2((0,1),\mathcal{M}_2)$. The proof is complete.
\end{proof}

Observe that the~$2\times 2$ matrix~$U_0=e^{-\la xJ}$ is a
solution of the system~$\ell(P_0)U=\la U$ with zero
potential~$P_0$. Therefore, in view of the last theorem, the
solution~$U$ of the problem~\eqref{eq:Dirac_eq} can be obtained
from~$U_0$ by means of the transformation
operator~$\mathscr{T}=\mathscr{R}+\mathscr{K}$,
where~$\mathscr{R}$ is an operator of multiplication
by~$\exp\{J\int_0^xp\}$ and~$\mathscr{K}$ is an integral operator
acting as follows
\[
    \mathscr{K}f(x)=\int_0^xf(x-2s)K(x,s)ds.
\]
This transformation operator also performs similarity of the
corresponding differential expressions, namely,
\[
    \ell(P)\mathscr{T}=\mathscr{T}\ell(P_0).
\]


\section{Asymptotics}\label{sec:asy}


In this section, we derive the asymptotics of eigenvalues and
eigenfunctions and the corresponding norming constants of the
problem~\eqref{eq:intr.spr} under the Dirichlet boundary
conditions~\eqref{eq:intr.b.c.}. We also obtain the factorization
of the characteristic function for the problem under study.

\subsection{Asymptotics of the eigenvalues}

Consider the vector~$\mathbf{u}(x,\la)=U(x,\la)
(1,0)^\mathrm{t}=\left(u_1(x,\la), u_2(x,\la)\right)^{\mathrm t}$.
In view of~\eqref{eq:U}, the second component~$u_2(x,\la)$ of the
vector~$\mathbf{u}(x,\la)$ is given by
\begin{equation*}
u_2(x,\la)=-\sin a(x)+\int_0^x\!
k_{11}(x,s)\sin(\la(x-2s))ds+\int_0^x\!k_{21}(x,s)\cos(\la(x-2s))ds.
\end{equation*}
Observe that the function~$u_2(x,\la)$ solves
equation~\eqref{eq:intr.spr} and satisfies the initial
condition~$u_2(0,\la)=0$. However,
\[
    u_2^{[1]}(0,\la)=\la(u_1(0,\la)+cu_2(0,\la))=\la,
\]
where~$c=(v-r)(0)$ (see Remark~\ref{rem:v-r_cont}).
Therefore~$\varphi(\la)=u_2(1,\la)/\la$ is the characteristic
function for the spectral
problem~\eqref{eq:intr.spr},~\eqref{eq:intr.b.c.}.

Further observe that~$u_2(1,\la)$ can be written as
\begin{equation}\label{eq:ch_f}
    u_2(1,\la)=\sin\left(\la-p_0\right) +\int_0^1\!f(s)e^{i\la(1-2s)}ds
\end{equation}
with~$f(s)=\frac{1}{2}\left[k_{21}(1,s)+k_{21}(1,1-s)-ik_{11}(1,s)+ik_{11}(1,1-s)\right]$;
recall also that~$p_0:=\int_0^1\!p(s)ds$. Let us make the change
of variables~$z:=\la-p_0$ and consider the function
\[
     \delta(z)=\sin z +\int_0^1\!\tilde{f}(s)e^{iz(1-2s)}ds,
\]
where~$\tilde{f}(s)=f(s)e^{ip_0(1-2s)}$.
Clearly,~$\delta(z)=u_2\left(1,z+p_0\right)$. By Theorem~4
of~\cite{LevOst:79}, the zeros of~$\delta(z)$ can be labelled
according to their multiplicities as~$z_n$,~$n\in\bZ$, so
that~$z_n=\pi n+\tilde{\la}_n$, where the
sequence~$(\tilde{\la}_n)_{n\in\bZ}$ belongs to~$\ell_2(\bZ)$.
Hence the zeros of the function~$u_2(1,\la)$ can be labelled
according to their multiplicities as~$\la_n$,~$n\in\bZ$, so that
\[
    \la_n=\pi n+p_0+\tilde{\la}_n.
\]
 In view of this asymptotics, all but
finitely many zeros of~$u_2(1,\la)$ are simple.

Next note that~$u_2(1,\la)$ is a characteristic function of the
operator~$\mathcal{D}(P)$ defined in Section~\ref{sec:red},
whence~$\la=0$ is a zero of~$u_2(1,\la)$ of order 1 (see
Lemma~\ref{lem:red.mult_0}). However, under assumption~(A) it is
not a zero of the characteristic function~$\varphi(\la)$.
Therefore the set of all the eigenvalues of the
problem~\eqref{eq:intr.spr},~\eqref{eq:intr.b.c.} coincides with
the set of zeros of the function~$u_2(1,\la)$ different from~$0$.
Thus the following theorem holds true.

\begin{theorem}
The eigenvalues of the
problem~\eqref{eq:intr.spr},~\eqref{eq:intr.b.c.} can be labelled
according to their multiplicities as~$\la_n$
with~$n\in\bZ^*:=\bZ\setminus\{0\}$  so that
\begin{equation}\label{eq:asym_la}
    \la_n=\pi n+p_0+\tilde{\la}_n
\end{equation}
with an~$\ell_2$-sequence~$(\tilde{\la}_n)$.
\end{theorem}

Next we construct the factorization of the characteristic
function~$\varphi(\la)$, which allows to determine~$\varphi(\la)$
via the eigenvalues of~\eqref{eq:intr.spr},~\eqref{eq:intr.b.c.}.
We use this factorization to derive the formula determining the
norming constants of~\eqref{eq:intr.spr},~\eqref{eq:intr.b.c.} via
two spectra of the equation~\eqref{eq:intr.spr} under two types of
boundary conditions (see \cite{Pro:2012pro}).

\begin{theorem}\label{thm:fact} Suppose that~$\la_n$,~$n\in\bZ^*$, are the eigenvalues of the spectral
problem~\eqref{eq:intr.spr},~\eqref{eq:intr.b.c.}. Then the
characteristic function~$\varphi(\la)$ can be factorized in the
following way:
\begin{equation*}
    \varphi(\la)=\left\{\begin{array}{ll}
      \mathrm{V.p.}\!\prod\limits_{\substack{{n=-\infty}\\{n\ne 0}}}^{\infty}\dfrac{\la_n-\la}{\pi n}, &\text{ if }  p_0\ne \pi l,\; l\in \bZ, \\
     (-1)^l \;\mathrm{V.p.}\!\prod\limits_{\substack{{n=-\infty}\\{n\ne 0}}}^{\infty}\dfrac{\la_n-\la}{\pi n}, &\text{ if }  p_0= \pi l,\; l\in \bZ.                  \end{array}\right.
\end{equation*}
\end{theorem}
\begin{proof}
Suppose firstly that~$p_0\ne \pi l$,~$l\in\bZ$. Observe that the
function~$u_2(1,\la)$ is given by~\eqref{eq:ch_f} and so it is of
exponential type~$1$. Recall also that, by
Lemma~\ref{lem:red.mult_0},~$\la=0$ is a zero of~$u_2(1,\la)$ of
order 1. Therefore by Hadamard factorization theorem (see e.g.\
\cite{You:2001}) we have\footnote{Here and hereafter all infinite
products and sums are understood in the principal value sense and
the symbol~$\mathrm{V.p.}$ will be omitted.}
\[
    u_2(1,\la)=\la e^{A\la+B}\prod_{\substack{{n=-\infty}\\{n\ne
    0}}}^{\infty}\left(1-\frac{\la}{\la_n}\right)e^{\frac{\la}{\la_n}},
\]
where~$A$ and~$B$ are some constants and~$\la_n$ are the zeros of
the function~$u_2(1,\la)$. In view of the asymptotic
distribution~\eqref{eq:asym_la},  the
series~$\sum\frac{\la}{\la_n}$ converges. Indeed,
\begin{align*}
   \sum\limits_{\substack{{n=-\infty}\\{n\ne
    0}}}^\infty\frac{1}{\la_n}&=\sum_{n=1}^\infty\left(\frac{1}{\la_n}+\frac{1}{\la_{-n}}\right)=\sum_{n=1}^\infty
    \frac{\pi^2
    n^2}{\la_n\la_{-n}}\cdot\frac{\la_n+\la_{-n}}{\pi^2n^2},
\end{align*}
and we note that the
series~$\sum_{n=1}^\infty\frac{\la_n+\la_{-n}}{\pi^2n^2}$ is
absolutely convergent and the
sequence~$\bigl(\frac{\pi^2n^2}{\la_n\la_{-n}}\bigr)$ is uniformly
bounded.  Therefore we can write
\[
    u_2(1,\la)=\la e^{A'\la+B}\prod_{\substack{{n=-\infty}\\{n\ne
    0}}}^{\infty}\biggl(1-\frac{\la}{\la_n}\biggr)
\]
with some constant~$A'$.

 To find the values~$A'$ and~$B$, consider the
ratio~$\frac{u_2(1,\la)}{\sin (\la-p_0)}$ and find its limits
along the ray~$\la=re^{i\theta}$,~$\theta\ne 0,\pi$. In view
of~\eqref{eq:ch_f} and a refined version of the Riemann--Lebesgue
lemma~\cite[Lemma 1.3.1]{Mar:77}, we have
\begin{equation}\label{eq:rel1}
    \frac{u_2(1,re^{i\theta})}{\sin (re^{i\theta}-p_0)}=1+o(1),\quad
    r\rightarrow \infty.
\end{equation}

Recall (see e.g.\ \cite{You:2001}) that the
function~$\sin(\la-p_0)$ can be factorized  as follows:
\[
    \sin(\la-p_0)=(\la-p_0)\prod_{\substack{{n=-\infty}\\{n\ne
    0}}}^{\infty}\frac{\nu_n-\la}{\pi n},
\]
where~$\nu_n$ are the zeros of~$\sin(\la-p_0)$, i.e.\ ~$\nu_n=\pi
n+p_0$. Therefore we have
\begin{align}\label{eq:rel2}
\frac{u_2(1,\la)}{\sin (\la-p_0)}&=\frac{\la
e^{A'\la+B}}{\la-p_0}\cdot\prod_{\substack{{n=-\infty}\\{n\ne
    0}}}^{\infty}\frac{\pi n}{\la_n}\cdot\frac{\la_n-\la}{\nu_n-\la}.
\end{align}

Let us show that~$A'=0$. If~$A$ were not~$0$, then one could
choose the direction~$\theta$ such that~$\mathrm{Re} A'
re^{i\theta}$ tends to infinity as~$r\rightarrow\infty$. Next note
that by Lemma~\ref{lem:prod1} given below the
product~$\prod\limits_{\substack{{n=-\infty}\\{n\ne
0}}}^{\infty}\frac{\pi n}{\la_n}$ is convergent and by
Lemma~\ref{lem:prod2} the
product~$\prod\limits_{\substack{{n=-\infty}\\{n\ne
0}}}^{\infty}\frac{\la_n-r e^{i\theta}}{\nu_n-re^{i\theta}}$
converges to~$1$ as~$r\rightarrow\infty$ and~$\theta\ne0,\pi$.
These arguments together with~\eqref{eq:rel1} and~\eqref{eq:rel2}
give a contradiction. Thus~$A'=0$ and
\[
e^B\prod_{\substack{{n=-\infty}\\{n\ne
    0}}}^{\infty}\frac{\pi n}{\la_n}=1,
\]
yielding that
\[
\varphi(\la)=\prod\limits_{\substack{{n=-\infty}\\{n\ne
0}}}^{\infty}\dfrac{\la_n-\la}{\pi n}.
\]

If~$p_0=\pi l$ for some $l\in\bZ$, then
\[
    \sin(\la-p_0)=(-1)^l\la \prod\limits_{\substack{{n=-\infty}\\{n\ne 0}}}^{\infty}\frac{\pi n-\la}{\pi
    n},
\]
and so
\[
    \frac{u_2(1,\la)}{\sin (\la-p_0)}=(-1)^le^{A'\la+B}\prod\limits_{\substack{{n=-\infty}\\{n\ne 0}}}^{\infty}\frac{\pi n}{\la_n}\prod\limits_{\substack{{n=-\infty}\\{n\ne 0}}}^{\infty}\frac{\la_n-\la}{\pi
    n-\la}.
\]
Using this,~\eqref{eq:rel1} and the arguments analogous to above,
we obtain that~$A'=0$
and~$e^B=(-1)^l\prod\limits_{\substack{{n=-\infty}\\{n\ne
0}}}^{\infty}\frac{\la_n}{\pi n}$. Thus for~$p_0=\pi l$
\[
    \varphi(\la)=(-1)^l\prod\limits_{\substack{{n=-\infty}\\{n\ne 0}}}^{\infty}\dfrac{\la_n-\la}{\pi n}.
\]
The proof is complete.
\end{proof}

\begin{lemma}\label{lem:prod1}
The product~$\prod\limits_{\substack{{n=-\infty}\\{n\ne
0}}}^{\infty}\frac{\la_n}{\pi n}$ is convergent.
\end{lemma}
\begin{proof}
We firstly prove the result analogous to Lemma~3.1.
of~\cite{Hry:2010} for sequences of complex numbers. For
every~$\varepsilon>0$,
set~$\delta=\delta(\varepsilon):=\sup\limits_{|z+1|\ge
\varepsilon}\bigl|\frac{z-\ln(1+z)}{z^2}\bigr|$. Note
that~$\delta$ is finite because the
function~$f(z):=\frac{z-\ln(1+z)}{z^2}$ is analytic in the
domain~$D:=\{z\in\bC\mid |1+z|\ge\varepsilon\}$ and tends to zero
as~$|z|\rightarrow\infty$. Using arguments analogous to those in
the proof of the mentioned lemma, we obtain that if~$(a_n)$ is a
sequence of complex numbers from~$\ell_2$ such that~$\sum a_n$
converges and~$|1+a_n|\ge \varepsilon$ for all~$n\in\bZ$, then
\[
    \biggl|\ln
    \prod_{-\infty}^{\infty}(1+a_n)\biggr|\le\biggl|\sum_{-\infty}^{\infty}a_n\biggr|+\delta\sum_{-\infty}^{\infty}a_n^2.
\]

Using arguments analogous to those in the proof of  one can show
that for each~$\varepsilon>0$, there
exists~$\delta=\delta(\varepsilon)$ such that whenever the
sequence~$(a_n)$ of numbers from~$\ell_2(\bZ)$ with~$|\sum
a_n|<\infty$ satisfies the inequality~$|a_n+1|>\varepsilon$ for
all~$n\in\bZ$, the following estimate is valid

Therefore, for~$\varepsilon$ such that~$\bigl|\frac{\la_n}{\pi
n}\bigr|>\varepsilon$
\[
    \biggl|\ln
    \prod\limits_{\substack{{n=-\infty}\\{n\ne
0}}}^{\infty}\frac{\la_n}{\pi
n}\biggr|\le\biggl|\sum_{\substack{{n=-\infty}\\{n\ne
0}}}^{\infty}\biggl(\frac{p_0}{\pi n}+\frac{\tilde\la_n}{\pi
n}\biggr)\biggr|+\delta\sum_{\substack{{n=-\infty}\\{n\ne
0}}}^{\infty}\biggl(\frac{p_0+\tilde\la_n}{\pi n}\biggr)^2
\]
with~$\delta:=\max\limits_{|z+1|>\varepsilon}\bigl|\frac{z-\ln(1+z)}{z^2}\bigr|$.

Observe that~$\mathrm{V.p.}\sum\frac{p_0}{\pi n}=0$. The
sequences~$(\tilde\la_n)$ and~$(1/n)$ are from~$\ell_2$, so that
the series~$\sum\frac{\tilde\la_n}{\pi n}$ is convergent; as a
result, the first summand in the righthand side of the last
estimate is finite. Since the sequence~$({1}/{\pi n})$ is
from~$\ell_2$ and~$\bigl(\frac{p_0+\tilde\la_n}{\pi
n}\bigr)^2\le\frac{C}{\pi^2n^2}$
with~$C:=\sup\limits_{n}(p_0+\tilde\la_n)^2$, the second summand
is finite as well. All these arguments imply that the
product~$\prod\limits_{\substack{{n=-\infty}\\{n\ne
0}}}^{\infty}\frac{\la_n}{\pi n}$ is convergent and so complete
the proof.
\end{proof}

\begin{lemma}\label{lem:prod2}
The product~$\prod\limits_{\substack{{n=-\infty}\\{n\ne
    0}}}^{\infty}\frac{\la_n-re^{i\theta}}{\nu_n-re^{i\theta}}$
converges to~$1$, as~$r\rightarrow\infty$ with~$\theta\ne 0,\pi$.
\end{lemma}
\begin{proof} Consider the series
\begin{equation}\label{eq:ser_ln}
    \sum\limits_{\substack{{n=-\infty}\\{n\ne
    0}}}^{\infty}\ln\frac{\la_n-re^{i\theta}}{\nu_n-re^{i\theta}}=\sum_{\substack{{n=-\infty}\\{n\ne
    0}}}^{\infty}\ln\biggl(1+\frac{\tilde{\la}_n}{\nu_n-re^{i\theta}}\biggr).
\end{equation}
One can find~$N$ sufficiently large such
that~$|\tilde{\la}_n|<1/2$ if~$|n|>N$. Also, for all~$r>R_\theta$
with~$R_{\theta}=(1+|p_0|)/\sin \theta$ we
have~$|\nu_n-re^{i\theta}|>1$. Since~$|\ln(1+z)|\le |z|$
if~$|z|\le 1/2$, for such~$n$ and~$r$
\[
    \biggl|\ln\biggl(1+\frac{\tilde{\la}_n}{\nu_n-re^{i\theta}}\biggr)\biggr|
    \le\biggl|\frac{\tilde{\la}_n}{\nu_n-re^{i\theta}}\biggr|.
\]
Next observe that
\[
    |\nu_n-re^{i\theta}|\ge |\pi n-r e^{i\theta}|-|p_0|\ge |\pi n\sin
    \theta| - |p_0|.
\]
Since~$|p_0|<\frac{1}{2}\pi N\sin \theta$ for sufficiently
large~$N$, for all~$n$ with~$|n|>N$ we have
\[
|\nu_n-re^{i\theta}|> \frac{|\sin \theta|}{2}|\pi n|.
\]
Therefore,
\[
    \biggl|\frac{\tilde{\la}_n}{\nu_n-re^{i\theta}}\biggr|<\frac{2}{\pi |\sin
    \theta|}\frac{|\tilde{\la}_n|}{|n|}.
\]
Since the sequences~$(\tilde{\la}_n)$ and~$\left(1/n\right)$
belong to~$\ell_2$, the series~$\sum\limits_{n=-\infty}^{\infty}
\frac{|\tilde{\la}_n|}{|n|}$ is convergent and so the
series~\eqref{eq:ser_ln} is  convergent uniformly
in~$r>R_{\theta}$ for a fixed~$\theta$,~$\theta\ne 0, \pi$.
Therefore,
\[
\lim\limits_{r\rightarrow \infty}
\sum\limits_{\substack{{n=-\infty}\\{n\ne
    0}}}^{\infty}\ln\frac{\la_n-re^{i\theta}}{\nu_n-re^{i\theta}}=
\sum\limits_{\substack{{n=-\infty}\\{n\ne
    0}}}^{\infty}\lim\limits_{r\rightarrow \infty}
\ln\frac{\la_n-re^{i\theta}}{\nu_n-re^{i\theta}}=0,
\]
which means that the
product~$\prod\limits_{\substack{{n=-\infty}\\{n\ne
    0}}}^{\infty}\frac{\la_n-re^{i\theta}}{\nu_n-re^{i\theta}}$
converges to~$1$ as~$r\rightarrow\infty$.
\end{proof}

\subsection{Asymptotics of eigenfunctions and norming constants}

Let us now consider the vectors~$\bu_n:=\bu(x,\la_n)$.
Put~$\bu_{n,0}:=(\cos \la_n x, \sin \la_nx)^{\mathrm t}$. Then, in
view of Theorem~\ref{thm:tr.op}, we have
\[
    \bu_n=\mathscr R\bu_{n,0}+\mathscr L\bu_{n,0},
\]
where the operator~$\mathscr R$ was defined at the end of
Section~\ref{sec:tr.op} and
\[
    \mathscr{L}\bu(x)=\int_0^xL(x,s)\bu(x-2s)ds
\]
with
\[
    L(x,s)=\left(%
\begin{array}{cc}
  k_{11}(x,s) & -k_{21}(x,s) \\
  k_{21}(x,s) & k_{11}(x,s) \\
\end{array}%
\right)
\]
and~$k_{ij}$ being the corresponding entries of~$K$. This yields
the following
\begin{theorem} The eigenfunctions~$y_n$ of the
problem~\eqref{eq:intr.spr},~\eqref{eq:intr.b.c.} corresponding to
the eigenvalues~$\la_n$ satisfy the asymptotics
\[
 y_n(x)=\sin \left(\la_nx-\int_0^x\!p\right)+\tilde{y}_n(x)
\]
with~$\tilde{y}_n(x)=(0,1)\mathscr L \bu_{n,0}$.
\end{theorem}
Note that the vectors~$\bv_n:=(\cos(\pi n+p_0)x,\sin(\pi
n+p_0)x)^{\mathrm{t}}$ form an orthonormal basis
in~$L_2(0,1)\times L_2(0,1)$. Next observe that
\begin{align*}
\left\|\bu_{n,0}-\bv_n\right\|&=\left\|(e^{\tilde{\la}_nxJ}-I)\bv_n\right\|\le\left\|\int_0^x\frac{d}{dt}e^{\tilde{\la}_ntJ}dt\right\|\le
C \|\tilde{\la}_n\|,
\end{align*}
where~$C:=\max_{n\in\bZ} e^{|\tilde{\la}_n|}$. Since the
sequence~$(\tilde{\la}_n)$ belongs to~$\ell_2$, this means that
the sequence~$\bu_{n,0}$ is quadratically close to the orthonormal
basis and therefore it is a Bari basis~\cite{You:2001}. Then
\begin{align*}
\|\bu_n-\mathscr R\bu_{n,0}\|&=\|\mathscr L\bu_{n,0}\|\le
\|\mathscr{L}(\bu_{n,0}-\bv_n)\|+\|\mathscr{L}\bv_n\|.
\end{align*}
Since~$\bu_{n,0}$ is quadratically close to~$\bv_n$, the
sequence~$(\|\mathscr{L}(\bu_{n,0}-\bv_n)\|)$ belongs to~$\ell_2$.
Next observe that the operator~$\mathscr{L}$ is of Hilbert-Schmidt
class. Therefore~$(\|\mathscr{L}\bv_n\|)$ also belongs to~$\ell_2$
(see e.g.\ \cite[V.2.4.]{Kat:1966}). All these arguments imply
that the sequence~$(\|\bu_n-\mathscr R\bu_{n,0}\|)$ is
from~$\ell_2$. Since
\[
    |\|\bu_n\|-1|\le\|\bu_n-\mathscr R\bu_{n,0}\|,
\]
we obtain that
\[
    \|\bu_n\|=1+\tilde{\al}_n,
\]
where~$(\tilde{\al}_n)$ belongs to~$\ell_2$.

 Observe also that if~$p$ and~$q$ are real-valued, the
norming constants of~\eqref{eq:intr.spr},~\eqref{eq:intr.b.c.}
corresponding to real and simple eigenvalues coincide with the
norms of eigenvectors of the operator~$\mathcal{D}(P)$ (see
\cite{HryPro:2012}). Therefore the following holds true.
\begin{theorem}
If~$p$ and~$q$ are real-valued, the norming constants~$\al_n$ of
the problem~\eqref{eq:intr.spr},~\eqref{eq:intr.b.c.}
corresponding to the eigenvalues~$\la_n$ satisfy the asymptotics
\[
    \al_n=1+\tilde{\al}_n,
\]
where~$(\tilde{\al}_n)\in\ell_2$.
\end{theorem}

\section{Results for the mixed boundary conditions}


In this section, we consider the equation~\eqref{eq:intr.spr}
under the boundary conditions which we call the mixed ones;
namely,
\begin{equation}\label{eq:mx_b.c.}
    y(0)=y^{[1]}(1)+h y(1)=0
\end{equation}
with some complex~$h$. As the proofs are analogous to those
applied in the case of the Dirichlet boundary conditions, we shall
only reformulate the results.

Without loss of generality, we assume that~$\mu=0$ is not an
eigenvalue of the problem \eqref{eq:intr.spr},~\eqref{eq:mx_b.c.}.
Consider the function
\[
    \psi(\mu):=u_1(1,\mu)+\frac{(h_1+h)u_2(1,\mu)}{\mu},
\]
where~$u_1$ and~$u_2$ are the solutions of the
system~\eqref{eq:Dir.system1},\eqref{eq:Dir.system2}
and~$h_1:=(v-r)(1)$. The function~$y(\cdot,\mu):=u_2(\cdot,\mu)$
solves the equation~$\ell(y)+2\mu py=\mu^2y$ and satisfies the
relation
\[
    y^{[1]}+hy=(y'-vy)+(v-r)y+hy=\mu u_1+(v-r+h)u_2;
\]
in particular,~$y(0,\mu)=0$
and~$y^{[1]}(1,\mu)+hy(1,\mu)=\mu\psi(\mu)$. Therefore~$\psi$ is a
characteristic function for the
problem~\eqref{eq:intr.spr},~\eqref{eq:mx_b.c.}, i.e. the zeros
of~$\psi(\mu)$ are the eigenvalues of the mentioned problem.

 Using
the asymptotics form of~$u_1$ and~$u_2$, we have
\begin{align*}
    \psi(\mu)=&\cos a(1,\mu)-\frac{(h_1+h)}{\mu}\sin a(1,\mu)\\&-\int_0^1
k_{21}(1,s)\sin(\mu(1-2s))ds+\int_0^1k_{11}(1,s)\cos(\mu(1-2s))ds\\&+\frac{h_1+h}{\mu}\left(\int_0^1
k_{11}(1,s)\sin(\mu(1-2s))ds+\int_0^1k_{21}(1,s)\cos(\mu(1-2s))ds\right).
\end{align*}
Taking into account Theorem 4 from~\cite{LevOst:79}, we obtain
\begin{theorem}
\label{thm:asy_mix} The eigenvalues
of~\eqref{eq:intr.spr},~\eqref{eq:mx_b.c.} can be labelled
according to their multiplicities as~$\mu_n$,~$n\in\bZ$, so that
they satisfy the following asymptotics
\begin{equation}\label{eq:asym}
    \mu_n=\pi \left(n+\frac12\right)+p_0+\tilde{\mu}_n,
\end{equation}
with~$\ell_2$-sequence~$(\tilde{\mu}_n)$. In particular, all
eigenvalues~$\mu_n$ with large enough~$|n|$ are simple.
\end{theorem}
Analogously to the case of the Dirichlet boundary conditions, we
prove the following
\begin{theorem}\label{thm:fact_mix}
Let~$\mu_n$,~$n\in\bZ$, be the eigenvalues
of~\eqref{eq:intr.spr},~\eqref{eq:mx_b.c.}. Then the
characteristic function~$\psi(\mu)$ can be factorized in the
following way
\[
\psi(\mu)=\left\{\begin{array}{ll}
                   -\mathrm{V.p.}\prod\limits_{n=-\infty}^{\infty}\dfrac{\mu_n-\mu}{\pi
\left(n+1/2\right)}, & \text{ if } p_0\ne \dfrac{\pi}{2}+\pi l,\; l\in \bZ, \\
                  (-1)^{l+1}(\mu_0-\mu)\mathrm{V.p.}\prod\limits_{\substack{{n=-\infty}\\{n\ne 0}}}^{\infty}\dfrac{\mu_n-\mu}{\pi n}, & \text{ if } p_0=\dfrac{\pi}{2}+\pi l,\; l\in\bZ. \\
                 \end{array}
\right.
\]
\end{theorem}

\begin{theorem}
The eigenfunctions~$y_n$
of~\eqref{eq:intr.spr},~\eqref{eq:mx_b.c.} corresponding to the
eigenvalues~$\mu_n$ satisfy the asymptotics
\[
 y_n(x)=\cos \left(\la_nx-\int_0^x\!p\right)+\tilde{y}_n(x),
\]
where the sequence~$(\|\tilde{y}_n(x)\|)$ is from~$\ell_2$.
\end{theorem}

\begin{theorem} If~$p$ and~$q$ are real-valued, the norming
constants~$\beta_n$ of~\eqref{eq:intr.spr},~\eqref{eq:mx_b.c.}
corresponding to the eigenvalues~$\mu_n$ satisfy the asymptotics
\[
    \beta_n=1+\tilde{\beta}_n,
\]
with an~$\ell_2$-sequence~$(\tilde{\beta}_n)$.
\end{theorem}

\bigskip

\noindent\emph{Acknowledgement.} The author is thankful to her
supervisor Dr.~Rostyslav Hryniv for valuable suggestions and for
help with the preparation of the manuscript.


\appendix

\section{Algebraic multiplicities of the eigenvalues}\label{sec:A}
In this appendix we recall the main notions of the spectral theory
for the operator pencils. We also show that the algebraic
multiplicity of~$\la$ as an eigenvalue of the operator pencil~$T$
of~\eqref{eq:intr.T} coincides with the corresponding multiplicity
of~$\la$ as an eigenvalue of the
problem~\eqref{eq:intr.spr},~\eqref{eq:intr.b.c.}.

An \emph{operator pencil}~$T$ is an operator-valued function
on~$\bC$. The \emph{spectrum} of an operator pencil~$T$ is the
set~$\sigma(T)$  of all $\lambda\in\bC$ such that~$T(\lambda)$ is
not boundedly invertible, i.e.\,
\[
    \sigma(T)=\{\lambda\in\mathbb{C}\mid 0\in\sigma(T(\lambda))\}.
\]
A number $\lambda\in\bC$ is called an \emph{eigenvalue} of $T$ if
$T(\lambda)y=0$ for some non-zero function~$y\in\dom T$, which is
then the corresponding \emph{eigenfunction}.

Vectors~$y_1,...,y_{m-1}$ from~$\dom T$ are said to be associated
with an eigenvector~$y_0$ correspon\-ding to an eigenvalue~$\la$ if
\begin{equation}\label{eq:A.chain}
    \sum\limits_{k=0}^j\frac{1}{k!}T^{(k)}(\la)y_{j-k}=0,\quad {j=1,...,m-1}.
\end{equation}
Here~$T^{(k)}$ denotes the \emph{k}-th derivative of~$T$ with
respect to~$\la$. The number~$m$ is called the length of the
chain~$y_0,\dots,y_{m-1}$ of an eigen- and associated vectors. The
maximal length of a chain starting with an eigenvector~$y_0$ is
called the \emph{algebraic multiplicity} of an eigenvector~$y_0$.

For an eigenvalue~$\la$ of~$T$ the dimension of the null-space
of~$T(\la)$ is called the \emph{geometric multiplicity} of~$\la$.
The eigenvalue is said to be \emph{geometrically simple} if its
geometric multiplicity equals to one.

All the eigenvalues of the pencil~$T$ of~\eqref{eq:intr.T} are
geometrically simple (see \cite{Pro:2012pro}), and then the
\emph{algebraic multiplicity} of an eigenvalue is the algebraic
multiplicity of the corresponding eigenvector. (If the
eigenvalue~$\la$ is not geometrically simple, its algebraic
multiplicity is the number of vectors in the corresponding
canonical system, see~\cite{Mar:88,Kel:1971}). An eigenvalue is
said to be \emph{algebraically simple} (or just \emph{simple}) if
its algebraic multiplicity is one.

 In the next proposition we show
that the order of~$\la$ as a zero of the characteristic
function~$\varphi$ coincides with the algebraic multiplicity
of~$\la$ as an eigenvalue of the operator pencil~$T$ defined
by~\eqref{eq:intr.T}.

\begin{proposition}\label{pro:pre.def_coin}
Suppose~$\la$ is an eigenvalue of the spectral
problem~\eqref{eq:intr.spr},~\eqref{eq:intr.b.c.}. Then~$\la$ is a
zero of the characteristic function~$\varphi$ of order~$m$ if and
only if~$\la$ is an eigenvalue of the operator pencil~$T$ given
by~\eqref{eq:intr.T} of algebraic multiplicity~$m$.
\end{proposition}
\begin{proof}
Suppose that~$y(x,z)$ is the solution of~\eqref{eq:intr.spr}
subject to the initial conditions~$y(0,z)=0$,~$y^{[1]}(0,z)=1$ and
that~$\la$ is a zero of~$\varphi(z)=y(1,z)$ of order~$m$.
Then~$y(x,\la)$ is an eigenfunction
of~\eqref{eq:intr.spr},~\eqref{eq:intr.b.c.} corresponding
to~$\la$. Clearly,~$y(x,\la)$ is also an eigenfunction of the
operator pencil~$T$ corresponding to the eigenvalue~$\la$.
 Consider the chain of the
vectors~$y_j$,~$j=0,1,\dots$, such that~$y_0=y(x,\la)$ and
\[
    y_{j}(x,\la):=\left.\frac{1}{j!}\frac{\partial^j
y(x,z)}{\partial z^j}\right|_{z=\la},\quad j\ge 1.
\]
Set
\[
    \tau(\la)y:=\la^2y-2\la py-\ell(y).
\]
Straightforward verification shows that~$y_j$ satisfy
equalities~\eqref{eq:A.chain} with~$\tau$ instead of~$T$.
Moreover, since~$\la$ is a zero of~$y(1,z)$ of order~$m$, we
have~$y_1(1)=\dots=y_{m-1}(1)=0$, and all the functions~$y_j$,
$j=0,\dots, m-1$, belong to the domain of~$T$ and so form a chain
of eigen- and associated vectors of~$T$ corresponding to~$\la$.
Therefore~$m$ does not exceed the algebraic multiplicity of~$\la$
as an eigenvalue of~$T$.

Assume that~$v_0,\dots, v_{l}$ is a chain of eigen- and associated
vectors corresponding to an eigenvalue~$\la$ of~$T$. Then~$v_0$
solves the equation
\[
    \tau(\la) y=0
\]
and satisfies the boundary conditions~\eqref{eq:intr.b.c.}, and
thus coincides with~$y_0$ up to a scalar factor. Without loss of
generality, we assume that~$v_0=y_0$ and then show by induction
that there exists a sequence~$(c_k)_{k=1}^l$ such that
\begin{equation}\label{eq:ind}
v_k-y_k=\sum_{j=1}^{k}c_{j}v_{k-j}.
\end{equation}
To start with, observe that
\[
    \tau'(\la)(v_0-y_0)+\tau(\la)(v_1-y_1)=0.
\]
But~$v_0-y_0=0$ and~$(v_1-y_1)(0)=0$. Therefore~$v_1-y_1=c_1v_0$
giving the base of induction. Next suppose that the statement
holds for all~$k<n$ and prove it for~$k=n$. Observe that
\begin{align*}
    \tau(\la)(v_n-y_n)&=-\sum\limits_{j=1}^{n}\frac{1}{j!}\tau^{(j)}(v_{n-j}-y_{n-j}).
\end{align*}
By assumption we obtain
\begin{align*}
\tau(\la)(v_n-y_n)&=-\sum\limits_{j=1}^n\frac{1}{j!}\tau^{(j)}\biggl(\sum\limits_{i=1}^{n-j}c_iv_{n-j-i}\biggr)\\
&=-\sum\limits_{i=1}^{n-1}c_{i}\biggl(\sum\limits_{j=1}^{n-i}\frac{1}{j!}\tau^{(j)}v_{n-j-i}\biggr)
=\sum\limits_{j=1}^{n-1}c_{j}\tau(\la)v_{n-j}.
\end{align*}
Therefore,
\[
    \tau(\la)\biggl(v_n-y_n-\sum\limits_{j=1}^{n-1}c_{j}v_{n-j}\biggr)=0
\]
and
\[
    \biggl(v_n-y_n-\sum\limits_{j=1}^{n-1}c_{j}v_{n-j}\biggr)(0)=0,
\]
giving that
\[
    v_n-y_n-\sum\limits_{j=1}^{n-1}c_{j}v_{n-j}=c_{n}v_0,
\]
i.e.~\eqref{eq:ind} holds.

Assuming that~$l\ge m$, we see
that~$v_m-y_m=\sum_{j=1}^mc_{j}v_{m-j}$ and so~$y_m(1)=0$. This
contradicts the fact that~$\la$ is a zero of~$\varphi(z)$ of
order~$m$. The proof is complete.
\end{proof}

\bibliographystyle{abbrv}

\bibliography{My_bibl}

\end{document}